\documentclass{article}

\usepackage{amsmath,amssymb,amsthm,bbm}
\providecommand{\R}{\mathbb R}
\providecommand{\E}{\mathsf E}
\providecommand{\prob}{\mathsf P}
\providecommand{\e}{\mathrm e}

\newtheorem{thm}{Theorem}
\newtheorem{lem}[thm]{Lemma}

\theoremstyle{remark}
\newtheorem*{rem}{Remark}
\newtheorem*{rems}{Remarks}

\title{Borell's formula on a Riemannian manifold and applications}
\author{Joseph Lehec}
\begin{document}
\maketitle
\section{Introduction}
Throughout the article $(\Omega,\mathcal A,\prob)$
is a probability space, equipped with 
a filtration $(\mathcal F_t )_{t\leq T}$
and carrying a standard $n$--dimensional 
Brownian motion $(B_t)_{t\leq T}$. 
The time horizon $T$ shall vary along the article, 
most of the time it will be finite.  
By \emph{standard} we mean that 
$(B_t)$ starts from $0$ and has quadratic covariation given 
by $[B]_t = t I_n$ for all $t\leq T$.  Let $\mathbb H$ be the Cameron-Martin space, namely 
the space of absolutely continuous paths $u\colon [0,T] \to \R^n$,
starting from $0$ and equipped with the norm 
\[
\Vert u \Vert_{\mathbb H} = \left(  \int_0^T \vert \dot u_s \vert^2 \, ds \right)^{1/2} ,
\]
where $\vert \dot u_s \vert$ denotes the Euclidean norm of the derivative of $u$ at time $s$.  
In this context a \emph{drift} is a process which is adapted to the filtration $(\mathcal F_t)$ 
and which belongs to $\mathbb H$ almost surely. Let $\gamma_n$ be the standard Gaussian measure on
$\R^n$. The starting point of the
present article is the so called Borell formula: 
If $f \colon \R^n \to \R$ is measurable and bounded from below then
\begin{equation}\label{eq:borell}
\log \left( \int_{\R^n} \e^{ f } \, d \gamma_n \right) 
= \sup_U \left\{ \E \left[ f ( B_1 + U_1 ) - \frac{1}{2} \Vert U \Vert^2_{\mathbb H} \right]  \right\}
\end{equation}
where the supremum is taken over all drifts $U$ 
(here the time horizon is $T=1$). 
Actually a more general formula holds true, where the function 
$f$ is allowed to depend on the whole trajectory of $(B_t)$ rather than 
just $B_1$. More precisely, let $(\mathbb W , \mathcal B , \gamma )$ 
be the $n$--dimensional Wiener space: $\mathbb W$ is the space of continuous 
paths from $[0,T]$ to $\R^n$, $\mathcal B$ is the Borel $\sigma$--field associated to the
topology given by the uniform convergence (uniform convergence on compact sets 
if $T=+\infty$) and $\gamma$ is the law of the standard Brownian motion. 
If $F \colon \mathbb W \to \mathbb R$ is measurable and bounded from below then 
\begin{equation}\label{eq:boue-dupuis}
\log \left( \int_{\mathbb W} \e^{ F } \, d\gamma \right) 
= \sup_U \left\{ \E \left[ F ( B + U ) - \frac{1}{2} \Vert U \Vert^2_{\mathbb H} \right]  \right\}
\end{equation}
Of course~\eqref{eq:borell} is retrieved by applying the latter formula to 
a functional $F$ of the form $F(w) = f ( w_1 )$. Formula~\eqref{eq:boue-dupuis} is
due to Bou\'e and Dupuis~\cite{boue-dupuis}, we refer to our previous work~\cite{lehec1} 
for more historical comments and an alternate proof of the formula. 
We are interested in the use of such formulas to prove functional 
inequalities. This was initiated by Borell in~\cite{borell}, in which he proved~\eqref{eq:borell}
and showed that it yields the Pr\'ekopa--Leindler inequality very easily. This was further developed 
in the author's works~\cite{lehec1,lehec2} where many other functional inequalities were derived
from~\eqref{eq:borell} and~\eqref{eq:boue-dupuis}. The main purpose of the present article is to establish a version 
of Borell's formula~\eqref{eq:borell} for the Brownian motion on a Riemannian manifold and to give 
a couple of applications, including a new proof of the Brascamp--Lieb inequality on the
sphere of Carlen, Lieb and Loss. 
\section{Borell's formula for a diffusion}
\label{sec:diffusion}
Let $\sigma \colon \R^n \to M_n ( \R)$, let 
$b\colon \R^n \to \R^n$ and assume that the 
stochastic differential equation
\begin{equation}\label{eq:SDE}
d X_t = \sigma ( X_t ) \, d B_t + b ( X_t ) \, dt 
\end{equation}
has a unique strong solution.
We also assume for simplicity that the 
explosion time is $+\infty$. 
Then there exists a measurable functional 
\[
G \colon \R^n \times \mathbb W \to \mathbb W
\]
(it is probably safer to complete the $\sigma$--field $\mathcal B$ 
at this stage) such that for every $x\in \R^n$ the process
$X = G ( x , B )$
is the unique solution of~\eqref{eq:SDE}
starting from $x$. 
This hypothesis is satisfied in particular 
if $\sigma$ and $b$ are locally Lipschitz
and grow at most linearly, 
see for instance~\cite[Chapter~IV]{ikeda}.
The process $(X_t)$ is then a diffusion 
with generator $L$ given by 
\[
L f = \frac 12 \langle \sigma \sigma^T , \nabla^2 f \rangle + \langle b , \nabla f \rangle . 
\]
for every $\mathcal C^2$--smooth function $f$. 
We denote the associated semigroup 
by $(P_t)$: for any test function $f$ 
\[
P_t f (x) = \E_x \left[ f ( X_t ) \right] ,
\]
where as usual the subscript $x$ denotes the starting point of $(X_t)$. 
Fix a finite time horizon $T$. Fix $x\in \R^n$, 
let $f \colon R^n \to \R$ and 
assume that $f$ is bounded from below. 
Applying the representation formula~\eqref{eq:boue-dupuis} to the functional 
\[
F \colon w \in \mathbb W \mapsto f \left( G ( x , w )_T \right) 
\]
we get
\[
\log \int_{\mathbb W} \e^{f \left( G (x,w)_T \right) } \, \gamma ( dw ) 
= \sup_U \left\{  \E \left[ f \left( G ( x, B+U )_T \right) - \frac 12 \Vert U \Vert^2_{\mathbb H} \right] \right\}
\]
where the supremum is taken on all drifts $U$. 
By definition of the semigroup $(P_t)$ we have
\[
\log \int_{\mathbb W} \e^{f \left( G (x,w)_T \right) }  \, \gamma (dw) = \log P_T ( \e^f ) (x) . 
\]
Also, we have the following lemma. 
\begin{lem}\label{lem:truc}
Let $(U_t)_{t\leq T}$ be a drift. 
The process $X^U = G(x,B+U)$ 
is the unique process satisfying 
\begin{equation}\label{eq:XtU}
X_t^U = x + \int_0^t \sigma ( X_s^U ) \, ( d B_s + d U_s ) + \int_0^t  b ( X_s^U ) \, ds ,  
\quad t\leq T 
\end{equation}
almost surely. 
\end{lem}
\begin{proof}
Assume first that $\Vert U \Vert_{\mathbb H}$ is
bounded. Then by Novikov's criterion the process $(D_t)$
given by
\[
D_t = \exp \left( - \int_0^t \langle \dot U_s , d B_s \rangle 
- \frac 12 \int_0^t \vert \dot U_s \vert^2 \, ds \right)
\]
is a uniformly integrable martingale. Moreover, 
according to Girsanov's formula, under the measure 
$\mathsf Q$ given by $d \mathsf Q = D_T \, d \mathsf P$,
the process $B+U$ is a standard Brownian motion on $[0,T]$, 
see for instance~\cite[section 3.5]{karatzas} for more details. 
Now since the stochastic differential equation~\eqref{eq:SDE}
is assumed to have a unique strong solution and by definition of $G$, 
almost surely for $\mathsf Q$, 
the unique process satisfying~\eqref{eq:XtU} is $X^U = G (x,B+U)$. 
Since $\mathsf Q$ and $\mathsf P$ are equivalent this is the result. 
For general $U$, the result follows by applying the 
bounded case to $(U_t) = ( U_{t\wedge T_n} )$ where $T_n$ is the 
stopping time
\[
T_n = \inf \left\{ t \geq 0 \colon \; \int_0^t \vert \dot U_s \vert^2 \, ds \geq n \right\} ,
\]
and letting $n$ tend to $+\infty$. 
\end{proof}
To sum up, we have established the following result. 
\begin{thm}\label{thm:diffusion}
For any function $f \colon \R^n \to \R$ bounded from 
below we have 
\[
\log P_T ( \e^f ) ( x ) =  
\sup_U \left\{  \E \left[ f ( X_T^U ) - \frac 12 \Vert U \Vert^2_{\mathbb H} \right] \right\} , 
\]
where the supremum is taken over all drifts $U$ 
and the process $X^U$ is the unique solution of~\eqref{eq:XtU}. 
\end{thm}
\begin{rems}
This direct consequence of the representation formula~\eqref{eq:boue-dupuis}
was already noted by Bou\'e and Dupuis. They used it 
to recover Freidlin and Wentzell's large deviation principle 
as the diffusion coefficient is sent to $0$. 
Let us also note that the non explosion hypothesis is not essential. 
One can consider $\R^n \cup \{\infty\}$, 
the one point compactification of $\R^n$, 
set $X_t = \infty$ after explosion time,
and restrict to functions $f$ that tend to 
$0$ at infinity. In the same way we could also 
deal with a Dirichlet boundary condition. 
\end{rems}
\section{Borell's formula on a Riemannian manifold}
\label{sec:manifold}
Let $(M,g)$ be a complete Riemannian manifold of dimension $n$.
In this section we wish to establish a Borell type formula
for the Brownian motion on $M$. To do so we need to recall 
first the intrinsic construction of the
Brownian motion on $M$. 
\\
Let us start with some definitions from 
Riemannian geometry. 
Recall that the orthonormal frame bundle 
$\mathcal O (M)$ is the set of 
$(n+1)$--tuples 
\[
\phi = ( x , \phi^1 , \dotsc , \phi^n ) 
\]
where $x$ is in $M$ 
and $(\phi^1,\dotsc,\phi^n)$ is
an orthonormal basis
of $T_x ( M )$.
Given an element $\phi = (x,\phi^1,\dotsc,\phi^n)$ of 
$\mathcal O (M)$ and a vector $v \in T_x(M)$, the 
\emph{horizontal lift}
of $v$ at $\phi$, denoted $\mathcal H (v)$, 
is an element of $T_\phi ( \mathcal O(M) )$ 
defined as follows:
Choose a curve $(x_t)$ starting from $x$
with speed $v$ and for $i\leq n$ let  
$\phi^i_t$ be the parallel translation of $\phi^i$ 
along $(x_t)$. Since parallel translation preserves
the inner product, we thus obtain a smooth curve $(\phi_t)$
on $\mathcal O (M)$ and we can set
\[
\mathcal H (v) = \dot \phi_0 . 
\]
This is a lift of $v$ in the sense 
that for any smooth $f$ on $M$ 
\[
\mathcal H (v) (f \circ \pi ) = v ( f) ,  
\]
where $\pi \colon \mathcal O (M) \to M$ is the 
canonical projection. 
Now for $i\leq n$ we define a vector 
field on $\mathcal O (M)$ by setting 
\[
\mathcal H^i ( x,\phi^1,\dotsc,\phi^n ) = \mathcal H (\phi^i)  .
\]
The operator
\[
\Delta_{\mathcal H} = \sum_{i=1}^n (\mathcal H^i)^2 
\]
is called the \emph{horizontal Laplacian}. It is 
related to the Laplace--Beltrami operator on $M$, 
denoted $\Delta$, through the following commutation 
property: for any smooth $f$ on $M$ we have 
\begin{equation}\label{eq:commute}
\Delta_{ \mathcal H } ( f \circ \pi ) = \Delta (f) \circ \pi .
\end{equation}
Note that the horizontal Laplacian is by definition a sum 
of squares of vector fields, and that this is typically not the 
case for the Laplace--Beltrami operator. We are now in a position 
to define the horizontal Brownian motion on $\mathcal O(M)$. 
Let
\[
B_t = ( B_t^1, \dotsc, B_t^n ) 
\]
be a standard $n$--dimensional Brownian motion. 
We consider the following stochastic differential equation on $\mathcal O (M)$
\begin{equation}\label{eq:hbm}
d \Phi_t = \sum_{i=1}^d \mathcal H^i ( \Phi_t ) \circ d B^i_t . 
\end{equation}
Throughout the rest of the article,
the notation $H \circ \, d M$
denotes the Stratonovitch integral. 
The equation~\eqref{eq:hbm} is a short way of saying that 
for any smooth function $g$ on $\mathcal O(M)$ we have 
\[
g( \Phi_t ) = g ( \Phi_0 ) + \sum_{i=1}^n \int_0^t \mathcal H^i ( g ) ( \Phi_t ) \circ d B^i_t . 
\]
This always has a strong solution, see~\cite[Theorem V.1.1.]{ikeda}. 
Let us assume additionally that it does not explode in finite 
time. This is the case in particular if the Ricci curvature of $M$ 
is bounded from below, see for instance~\cite[section~4.2]{hsu}, 
where a more precise criterion is given. 
Translating the equation above
in terms of It\^o increments we easily get
\[
d g ( \Phi_t ) 
 = \sum_{i=1}^n \mathcal H^i (g) ( \Phi_t ) \, d B^i_t + \frac 12 \Delta_{\mathcal H} g ( \Phi_t ) \, dt.   
\]
Let $(X_t)$ be the process given by $X_t = \pi ( \Phi_t )$. Applying the 
previous formula and~\eqref{eq:commute} we see that for any smooth $f$ 
on $M$
\begin{equation}\label{eq:ftruc}
d f (X_t) = \sum_{i=1}^n \Phi^i_t (f) (X_t) \, d B^i_t  + \frac 12 \Delta f ( X_t ) \, dt.
\end{equation}
In particular 
\[
f ( X_t ) - \frac 12 \int_0^t \Delta f( X_s ) \, ds , \quad t \geq 0 
\]
is a local martingale. This shows that $(X_t)$ is a Brownian motion on $M$. 
The process $(X_t)$ is called the \emph{stochastic development}
of $(B_t)$. In the sequel, it will be convenient to identify 
the orthogonal basis $\Phi^1_t,\dotsc,\Phi^n_t$ with the 
orthogonal map 
\[
x \in \R^n \to \sum_{i=1}^n x_i \Phi^i_t \in T_{X_t} (M).
\] 
Then the equation~\eqref{eq:ftruc} can be rewritten 
\[
d f ( X_t ) =  \langle \nabla f ( X_t ) , \Phi_t \, d B_t \rangle + \frac 12 \Delta f ( X_t ) \, dt.   
\]
Similarly the equation~\eqref{eq:hbm}
can be rewritten in a more concise form
\begin{equation}\label{eq:hbm2}
d \Phi_t = \mathcal H ( \Phi_t ) \circ d B_t  . 
\end{equation}
To sum up, the process $(\Phi_t)$ is an orthonormal 
basis above $(X_t)$ which is used to map the Brownian increment 
$d B_t$ from $\R^n$ to the tangent space of $M$ at $X_t$.
\\ 
Now we establish a Borell type formula for the process
$(X_t)$. We know that there exists
a measurable functional 
\[
G \colon \mathcal O (M) \times \mathbb W \to \mathcal C ( \mathbb R , \mathcal O (M) )
\]
such that the process $\Phi = G ( \phi , B)$ 
is the unique solution of~\eqref{eq:hbm2} starting from $\phi$.  
Let $\phi \in \mathcal O(M)$, let $T>0$, let $f\colon M \to \R$ and 
assume that $f$ is bounded from below. Applying the representation 
formula~\eqref{eq:boue-dupuis} to the functional 
\[
F \colon w \in \mathbb W \mapsto f \circ \pi \left( G ( \phi , w )_T \right)
\]
we get
\[
\log \left( \int_{\mathbb W} \e^{ f \circ \pi ( G(\phi,B)_T ) } \, d\gamma \right)  = 
\sup_U \left\{ \E \left[ f \circ \pi ( G( \phi , B+U )_T ) - \frac 12 \Vert U \Vert^2_{\mathbb H} \right] \right\} .
\]
Let $x = \pi ( \phi )$. Since $\pi ( G ( \phi , B ) )$ is a Brownian motion 
on $M$ starting from $x$ we have 
\[
\log \left( \int_{\mathbb W} \e^{ f \circ \pi ( G(\phi,B)_T ) } \, d\gamma \right) 
= \log P_T ( \e^f ) (x) , 
\]
where $(P_t)$ is the heat semigroup on $M$. Also, letting $\Phi^U = G ( \phi , B+U)$ and
reasoning along the same lines as in the proof of Lemma~\ref{lem:truc}, we obtain that
$\Phi^U$ is the only solution to 
\[
d \Phi^U_t = \mathcal H ( \Phi^U_t ) \circ ( d B_t + d U_t )  
\]
starting from $\phi$. We also let $X^U = \pi ( \Phi^U )$
and call this process the stochastic development of $B+U$ 
starting from $\phi$. 
To sum up, we have established the following result. 
\begin{thm}\label{thm:manifold}
Let $f \colon M \to \R$, let $\phi\in \mathcal O(M)$, let $x=\pi (\phi)$
and let $T >0$. If $f$ is bounded from below then 
\[
\log P_T \left( \e^f \right) (x)
= \sup_{U} \left\{  \E \left[ f( X_T^U ) - \frac 12 \Vert U \Vert^2_{\mathbb H} \right] \right\} ,
\]
where the supremum is taken on all drifts $U$ and where given a drift $U$,
the process $X^U$ is the stochastic development of $B+U$ starting from $\phi$. 
\end{thm}
\section{Brascamp--Lieb inequality on the sphere}
In the article~\cite{lehec2}, we explained how to derive the Brascamp--Lieb inequality
and its reversed version from Borell's formula. In this 
section we extend this to the sphere, and give a proof based on Theorem~\ref{thm:manifold}
of the spherical version of the Brascamp--Lieb inequality, 
due to Carlen, Lieb and Loss in~\cite{carlen}. 
\begin{thm}\label{thm:bl-sphere}
Let $g_1,\dotsc,g_{n+1}$ be non--negative 
functions on the interval $[-1,1]$. Let $\sigma_n$ 
be the Haar measure on the sphere $\mathbb S^n$, normalized 
to be a probability measure. We have
\[
\int_{\mathbb S^n} \prod_{i=1}^{n+1} g_i (x_i) \, \sigma_n (dx)
\leq \prod_{i=1}^{n+1} \left( \int_{\mathbb S^n} g_i(x_i)^2 \, \sigma_n (dx) \right)^{1/2}
\]
\end{thm}
\begin{rem}
The inequality does not hold if we replace the $L^2$ norm in the right--hand side
 by a smaller $L^p$ norm, like the $L^1$ norm. 
Somehow this $2$ accounts for the fact that the coordinates of a uniform 
random vector on $\mathbb S^n$ are not independent. We refer to the introduction of~\cite{carlen}
for a deeper insight on this inequality. 
\end{rem}
In addition to the Borell type formulas 
established in the previous two sections, our
proof relies on a sole inequality, 
spelled out in the lemma below. 
Let $P_i \colon \mathbb S^n \to [-1;1]$
be the application that maps $x$ to its $i$--th coordinate $x_i$. 
The spherical gradient of $P_i$ at $x$ is the projection of the coordinate 
vector $e_i$ onto $x^\perp$:
\[
\nabla P_i (x) = e_i - x_i x. 
\] 
\begin{lem}\label{lem:frame}
Let $x\in \mathbb S^n$ and let $y\in x^\perp$. For $i\leq n+1$, 
if $\nabla  P_i ( x ) \neq 0$ let
\[
\theta^i = \frac { \nabla P_i ( x ) }{ \vert \nabla P_i ( x ) \vert } 
\]
and let $\theta_i$ be an arbitrary unit vector of $x^\perp$ otherwise. 
Then for any $y\in x^\perp$ we have 
\[
\sum_{i=1}^{n+1} \langle \theta^i , y \rangle^2 \leq 2 \vert y \vert^2 .
\]
\end{lem}
\begin{proof}
Assume first that $\nabla P_i (x) \neq 0$ for every $i$. 
Since $\nabla P_i (x) = e_i - x_i x$ and $y$ is orthogonal to $x$
we then have 
\[
\langle \theta^i , y \rangle^2 
= y_i^2 + x_i^2 \langle \theta^i , y \rangle^2
\leq y_i^2 + x_i^2 \vert y \vert^2 .
\]
Summing this over $i\leq n+1$ yields the result. On the other hand, 
if there exists $i$ such that $\nabla P_i (x)= 0$ then $x = \pm e_i$
and it is almost immediate to check that the desired inequality holds true. 
\end{proof}
\begin{proof}[Proof of Theorem~\ref{thm:bl-sphere}]
Let us start by describing the behaviour of a given 
coordinate of a Brownian motion on $\mathbb S^n$.  
Let $(B_t)$ be a standard Brownian motion on $\R^n$,
let $\phi$ be a fixed element of $\mathcal O ( \mathbb S^n)$ 
and let $(\Phi_t)$ be the horizontal Brownian motion given by
\[
\Phi_0 = \phi  \quad \text{and} \quad d \Phi_t = \mathcal H ( \Phi_t ) \circ d B_t . 
\]
We also let $X_t = \pi ( \Phi_t )$ be the stochastic development 
of $(B_t)$ and $X_t^i = P_i ( X_t)$, for every $i\leq n+1$. We have 
\begin{equation}
\label{eq:jacobi}
d X_t^i = \langle \nabla P_i ( X_t ) , \Phi_t d B_t \rangle + \frac 1 2  \Delta P_i ( X_t ) \, dt . 
\end{equation}
Let $\theta$ be an arbitrary unit vector of $\R^n$ 
and define a process $(\theta^i_t)$ by
\[
\theta^i_t = 
\left\{
\begin{array}{ll}
\Phi_t^* \left( \frac { \nabla P_i ( X_t ) }{ \vert \nabla P_i ( X_t ) \vert } \right) ,
& \text{if }  \nabla  P_i (X_t) \neq 0 , \\
\theta & \text{otherwise.} 
\end{array}
\right. 
\]
Since $\Phi_t$ is an orthogonal map $\theta_t$
belongs to the unit sphere of $\R^n$. 
Consequently the process $(W^i_t)$ defined by 
$d W^i_t = \langle \theta^i_t , d W_t \rangle$
is a one dimensional standard Brownian motion.  
Observe that $\vert \nabla P_i \vert = (1 - P_i^2 )^{1/2}$ and 
recall that $P_i$ is an eigenfunction for the spherical Laplacian: 
$\Delta P_i = - n P_i$. Equality~\eqref{eq:jacobi} becomes
\[
d X^i_t = \left( 1- (X^i_t)^2 \right)^{1/2} \, d W^i_t - \frac n 2 \, X^i_t \, dt .  
\]
This stochastic differential equation is usually referred to 
as the Jacobi diffusion in the literature, see for instance~\cite[section 2.7.4]{bakry}. 
What matters for us is that it does possess a unique strong solution. 
Indeed the drift term is linear and although the diffusion factor $(1-x^2)^{1/2}$ 
is not locally Lipschitz, it is H\"older continuous with exponent $1/2$,
which is sufficient to insure strong uniqueness in dimension $1$, 
see for instance~\cite[section V.40]{rogers}. 
Let $(Q_t)$ be the semigroup associated to the process
$(X^i_t)$. The stationary measure $\nu_n$ is easily seen to be given by
\[
\nu_n (dt) = c_n \mathbbm 1_{[-1,1]} (t) (1-t^2)^{\frac n 2 - 1 } \,  dt ,
\]
where $c_n$ is the normalization constant. Obviously
$\nu_n$ coincides with the pushforward of $\sigma_n$ by $P_i$. 
\\
We now turn to the actual proof of the theorem. Let $g_1,\dotsc,g_{n+1}$
be non negative functions on $[-1,1]$ and assume (without loss of generality)
that they are bounded away from $0$. Let $f_i = \log ( g_i )$ for all $i$ and let 
\[
f \colon x\in \mathbb S^n \mapsto \sum_{i=1}^{n+1} f_i ( x_i ) .
\]
The functions $f_i,f$ are bounded from below. 
Fix a time horizon $T$, let $U$ be a drift and let $(\Phi_t^U)$ be the process given by 
\[
\left\{
\begin{array}{l}
\Phi^U_0 = \phi \\
d \Phi^U_t  = \mathcal H ( \Phi^U_t ) \circ ( d B_t + d U_t ) , \quad t\leq T. 
\end{array}
\right.
\]
We also let $X_t^U = \pi ( \Phi^U_t )$ be the stochastic development of $B+U$. 
These processes are well defined by the results of the previous section. 
We want to bound $f ( X_T^U ) - \frac 12 \Vert U \Vert_{\mathbb H}^2$ from above. 
By definition 
\[
f ( X_T^U ) = \sum_{i=1}^{n+1} f_i ( P_i  X_T^U ) ) . 
\]
Let $(\theta^i_t)$ be the process given by 
\[
\theta^i_t = \Phi_t^* \left( \frac { \nabla P_i ( X^U_t ) }{ \vert \nabla P_i ( X^U_t ) \vert } \right) 
\]
(again replace this by an arbitrary fixed unit vector if $\nabla P_i (X^U_t) = 0$). 
Then let $(W^i_t)$ be the one dimensional Brownian motion given by 
$d W^i_t = \langle \theta^i_t , d W_t \rangle$
and let $(U^i_t)$ be the one dimensional drift given by 
$d U^i_t = \langle \theta_t^i , d U_t \rangle$. 
The process $(P_i (X_t^U))$ then satisfies 
\begin{equation}\label{eq:jacobi-drift}
d P_i ( X^U_t) = \left( 1 - P_i ( X^U_t )^2 \right)^{1/2} \, \left( d W^i_t + d U^i_t \right) 
 - \frac n 2 \, P_i ( X^U_t ) \, dt . 
\end{equation}
Applying Lemma~\ref{lem:frame}, we easily get
\[
\sum_{i=1}^{n+1} \Vert U^i \Vert^2_{\mathbb H} \leq 2 \Vert U \Vert_{\mathbb H}^2 , 
\] 
almost surely 
(note that in the left hand side of the
inequality $\mathbb H$ is the Cameron--Martin space 
of $\R$ rather than $\R^n$). 
Therefore 
\begin{equation}\label{eq:step}
f ( X_T^U ) - \frac 12 \Vert U \Vert_{\mathbb H}^2 
 \leq \sum_{i=1}^{n+1} \left( f_i ( P_i ( X_T^U ) ) - \frac 14 \Vert U^i \Vert^2_{\mathbb H} \right) .
\end{equation}
Recall~\eqref{eq:jacobi-drift} and apply Theorem~\ref{thm:diffusion} 
to the semigroup $(Q_t)$ and to the function $2 f_i$ rather than $f_i$. 
This gives
\[
\E \left[ f_i ( P_i ( X^U_t ) ) - \frac 14 \Vert U^i \Vert^2_{\mathbb H} \right] 
\leq \frac 12 \log Q_T ( \e^{ 2 f_i } ) ( x_i ) ,
\]
for every $i\leq n+1$. 
Taking expectation in~\eqref{eq:step} thus yields
\[
\E \left[ f ( X_T^U ) - \frac 12 \Vert U \Vert_{\mathbb H}^2 \right]
\leq \frac 12 \sum_{i=1}^{n+1} \log Q_T \left( \e^{2f_i} \right) ( x_i ) . 
\]
Taking the supremum over all drifts $U$ and using 
Theorem~\ref{thm:manifold} we finally obtain
\[
P_T ( \e^f ) (x) \leq \prod_{i=1}^{n+1} \left( Q_T ( \e^{2 f_i} ) (x_i) \right)^{1/2}. 
\]
The semigroup $(P_t)$ is ergodic and converges to $\sigma_n$ 
as $t$ tends to $+\infty$. Similarly $(Q_t)$ converges to $\nu_n$. 
So letting $T$ tend to $+\infty$ in the previous inequality gives
\[
\int_{\mathbb S^n} \e^f \, d \sigma_n
\leq \prod_{i=1}^{n+1} \left( \int_{[-1,1]} \e^{2 f_i}  \, d \nu_n  \right)^{1/2} , 
\]
which is the result. 
\end{proof}
\begin{rem}
Barthe, Cordero--Erausquin and Maurey in~\cite{barthe1} and together 
with Ledoux in~\cite{barthe2} gave several extensions of Theorem~\ref{thm:bl-sphere}.
The method exposed here also allows to recover most of their results. We
chose to stick to the original statement of Carlen, Lieb and Loss 
for simplicity. 
\end{rem}
\section{The dual formula}
In~\cite{lehec1}, we established a dual version 
of Borell's formula~\eqref{eq:borell}. It states as follows: 
If $\mu$ is an absolutely continuous measure on $\R^n$ 
satisfying some (reasonable) technical assumptions, the relative 
entropy of $\mu$ with respect to the Gaussian measure is given by 
the following formula
\begin{equation}\label{eq:borell-dual}
\mathrm H ( \mu \mid \gamma_n ) = 
\inf \left\{ \frac 12 \E\left[ \Vert U \Vert^2_{\mathbb H} \right] \right\}  
\end{equation}
where the infimum is taken over all drifts $U$ such that $B_1+U_1$ 
has law $\mu$. Informally this says that the minimal energy needed to 
constrain the Brownian motion to have a prescribed law at time $1$ 
coincides with the relative entropy of this law with respect to $\gamma_n$. 
The infimum is actually a minimum
and the optimal
drift can be described as follows. 
We let $f$ be the density of $\mu$ 
with respect to $\gamma_n$ and $(P_t)$
be the heat semigroup on $\R^n$. 
The following stochastic differential equation 
\begin{equation}\label{eq:optimal1}
\left\{
\begin{array}{l}
X_0 = 0 \\
d X_t  = d B_t + \nabla \log P_{1-t} f ( X_t ) \, dt, \quad t\leq 1  .
\end{array}
\right. 
\end{equation}
has a unique strong solution. The solution 
satisfies $X_1 = \mu$ in law  
and is optimal in the sense that there
is equality in~\eqref{eq:borell-dual}
for the drift $U$ given by $\dot U_t = \nabla \log P_{1-t} f ( X_t )$. 
The purpose of this section is to describe
the Riemannian counterpart of~\eqref{eq:borell-dual}. 
There is also a version for diffusions such as the ones 
considered in section~\ref{sec:diffusion} but we shall omit
it for the sake of brevity.

The setting of this section is the same as that of
section~\ref{sec:manifold}: $(M,g)$ is a complete 
Riemannian manifold of dimension $n$ 
whose Ricci curvature is bounded from below and $(B_t)$ is a standard 
Brownian motion on $\R^n$. We denote the heat semigroup on $M$ 
by $(P_t)$. 
\begin{thm}
Fix $x\in M$ and a time horizon $T$. 
Let $\mu$ be a probability measure
on $M$, assume that $\mu$ is absolutely 
continuous with respect to $\delta_x P_T$
and let $f$ be its density. If $f$ 
is Lipschitz and bounded away from $0$ 
then 
\[
\mathrm H ( \mu \mid \delta_x P_T ) = 
\inf \left\{ \frac 12 \E \left[ \Vert U \Vert_{\mathbb H}^2 \right] \right\}
\]
where the infimum is taken all drifts $U$ such that 
the stochastic development of $B+U$ starting from $x$
has law $\mu$ at time $T$.
\end{thm} 
Proving that any drift satisfying the constraint has energy at least as large
as the relative entropy of $\mu$ is a straightforward adaptation of 
Proposition~1 from~\cite{lehec1}, and we shall leave this to the reader. 
Alternatively, one can use Theorem~\ref{thm:manifold} and combine it with the 
following variational formula for the entropy: 
\[
\mathrm H ( \mu \mid \delta_x P_T ) = \sup_f 
\left\{ \int_M f \, d \mu - \log P_T ( \e^f ) (x) \right\} 
\]
(again details are left to the reader). 
Besides, as in the Euclidean case, there is actually an optimal drift, 
whose energy is exactly the relative entropy of $\mu$. 
This is the purpose of the next result. 
\begin{thm}
Let $\phi$ be a fixed element of $\mathcal O(M)$. 
Let $x=\pi( \phi)$ and let $T$ be a time horizon. 
Let $\mu$ have density $f$ with respect to $\delta_x P_T$,
and assume that $f$ is Lipschitz and bounded away from $0$. 
The stochastic differential equation 
\begin{equation}\label{eq:optimal2}
\left\{
\begin{array}{l}
\Phi_0 = \phi \\
d \Phi_t  = \mathcal H ( \Phi_t ) \circ 
\left( d B_t + \Phi_t^* \nabla \log P_{T-t} f ( Y_t ) \right) \\
Y_t = \pi ( \Phi_t ) 
\end{array}
\right.
\end{equation}
has a unique strong solution on $[0,T]$. 
The law of the process $(Y_t)$
is given by the following formula:
For every functional  
$H \colon \mathcal C ( [0,T] ; M ) \to \R$ 
we have
\begin{equation}\label{eq:lawY}
\E \left[ H ( Y ) \right] = \E \left[ H(X) f ( X_1 ) \right] ,
\end{equation}
where $(X_t)$ is a Brownian motion on $M$
starting from $x$. 
In other words $Y_T$ has law $\mu$ and the bridges of $Y$
equal those of the Brownian motion on $M$, in law. 
Moreover, letting $U$ be the drift given by 
\begin{equation}\label{eq:defU}
U_t = \int_0^t \Phi_s^* \nabla \log P_{T-s} f ( Y_s)  \, ds , \quad t \leq T,
\end{equation} 
we have 
\[
\mathrm H ( \mu \mid \delta_x P_T ) =  \frac 12 \E \left[ \Vert U \Vert_{\mathbb H}^2 \right] . 
\]
\end{thm}
\begin{proof} 
Since $\mathrm{Ric} \geq - \lambda\,g$, we have the following estimate for the 
Lipschitz norm of $f$: 
\[
\Vert P_t f \Vert_{\rm Lip} \leq \e^{\lambda t /2 } \Vert f \Vert_{\rm Lip} . 
\]
One way to see this is to use Kendall's coupling for Brownian 
motions on a manifold, see for instance~\cite[section 6.5]{hsu}. 
Alternatively, it is easily derived from the commutation property
$\vert \nabla P_t f \vert^2 \leq \e^{\lambda t} P_t (\vert \nabla f \vert^2 )$
which, in turn, follows from Bochner's formula, see~\cite[Theorem 3.2.3]{bakry}. 
Recall that $f$ is assumed to be bounded away from $0$, and for 
every $t\leq T$ let $F_t = \log P_{T-t} f$. 
Then $(t,x)\mapsto \nabla F_t (x)$ is smooth and bounded
on $[0,T[ \times M$, which is enough to insure the existence of a 
unique strong solution 
to~\eqref{eq:optimal2}. Besides an easy computation shows that
\begin{equation}\label{eq:partialtF}
\partial_t F_t = - \frac 12 ( \Delta F_t + \vert \nabla F_t \vert^2 ) . 
\end{equation}
Then using~\eqref{eq:optimal2} and It\^o's formula we get
\[
\begin{split}
d F(t,Y_t) & = \langle \nabla F_t (Y_t ) , \Phi_t d B_t \rangle + 
\frac 12 \vert \nabla F_t (Y_t ) \vert^2 \, dt. \\
& = \langle \dot U_t	 , d B_t \rangle + 
\frac 12 \vert  \dot U_t \vert^2 \, dt 
\end{split}
\]
(recall the definition~\eqref{eq:defU} of $U$).
Therefore 
\begin{equation}\label{eq:stepsomething}
\frac 1 {f ( Y_T )} = \e^{ - F_T (Y_T)} 
= \exp \left(  - \int_0^T \langle \dot U_t , d B_t \rangle
- \frac 12 \Vert U \Vert_{\mathbb H}^2 \right) . 
\end{equation}
Observe that the variable $\Vert U \Vert_{\mathbb H}$ is bounded
(just because $\nabla F$ is bounded). 
So Girsanov's formula applies: $1/f(Y_T)$ has expectation $1$ 
and under the measure $\mathsf Q$ 
given by $d \mathsf Q = ( 1/ f(Y_T) ) \, d \mathsf P$
the process $B+U$ is a standard Brownian motion on $\R^n$. 
Since $Y$ is the stochastic development of
$B+U$ starting from $x$, this shows that under $\mathsf Q$
the process $Y$ is a Brownian motion on $M$ starting from 
$x$. This is a mere reformulation of~\eqref{eq:lawY}. 
For the entropy equality observe that since $Y_T$ 
has law $\mu$, we have 
\[
\mathrm H ( \mu \mid \delta_x P_T ) = \E [ \log f ( Y_T ) ] . 
\]
Using~\eqref{eq:stepsomething} again and the fact that 
$\int \langle \dot U_t ,d B_t \rangle$ is a martingale 
we get the desired equality. 
\end{proof}
To conclude this article, let us derive from this formula
the log--Sobolev inequality for a manifold 
having a positive lower bound on its Ricci curvature. 
This is of course well--known, but our point 
is only to illustrate how the previous theorem can be used to prove inequalities. 
Recall the definition of the Fisher information: if $\mu$ is a probability 
measure on $M$ having Lipschitz and positive density $f$ with respect to some 
reference measure $m$, the relative Fisher information of $\mu$ with respect to
$m$ is defined by
\[
\mathrm I ( \mu \mid m ) 
= \int_M \frac{ \vert \nabla f \vert^2 } f  \, dm 
= \int_M  \vert \nabla \log f \vert^2   \, d \mu . 
\]
By Bishop's Theorem, if $\mathrm{Ric} \geq \kappa \, g$ pointwise
for some positive $\kappa$ then the volume measure on $M$ 
is finite. We let $m$ be the volume measure normalized to be a probability measure. 
\begin{thm}
If $\mathrm {Ric} \geq \kappa \, g$ 
pointwise for some $\kappa >0$,
then for any probability 
measure $\mu$ on $M$ having 
a Lipschitz and positive density with respect to $m$ we have
\begin{equation}\label{eq:logsob1}
\mathrm H ( \mu \mid m ) 
\leq \frac n 2  \log \left( 1 + \frac { \mathrm I ( \mu \mid m )}{ n \, \kappa } \right) . 
\end{equation}
\end{thm}
\begin{rems}
Since $\log (1+x) \leq x$ this inequality is a dimensional improvement of the
more familiar inequality 
\begin{equation}\label{eq:logsob0}
\mathrm H ( \mu \mid m  ) \leq \frac 1 \kappa \, \mathrm I ( \mu \mid m ) . 
\end{equation}
It is known (see for instance~\cite[section 5.6]{ane})
that~\eqref{eq:logsob0} admits yet another sharp form
that takes the dimension into account, namely 
\begin{equation}\label{eq:logsob2}
\mathrm H ( \mu \mid m  ) \leq \frac {n-1} {\kappa\, n} \, \mathrm I ( \mu \mid m ) , 
\end{equation}
but we were not able to recover this one with our method. 
Note also that depending on the measure $\mu$, the right hand side of~\eqref{eq:logsob2} 
can be smaller than that of~\eqref{eq:logsob1}, or the other way around. 
\end{rems}
\begin{proof}
By the  Bonnet--Myers theorem $M$ is compact.  Fix $x\in M$ and a time horizon $T$.
Let $p_T ( x, \cdot )$ be the density of the measure $\delta_x P_T$ with respect to
$m$ (in other words let $(p_t)$ be the heat kernel on $M$).  
If $d \mu = \rho\, d m$ then $\mu$ has density $f = \rho / p_T(x, \cdot)$ with respect to 
$\delta_x P_T$. Since $p_T ( x ,\cdot )$ is smooth and positive
(see for instance~\cite[chapter 6]{chavel}) $f$ satisfies the technical assumptions
of the previous theorem. Let $F_t =\log P_{T-t} f$ and let $(Y_t)$ be the process 
given by~\eqref{eq:optimal2}.
We know from the previous theorem that
\begin{equation}\label{eq:entropylogsob}
\mathrm H ( \mu \mid \delta_x P_T ) =
\frac 12\, \E \left[ \int_0^T \vert \nabla F_t (Y_t) \vert^2 \, dt \right] . 
\end{equation}   
Using~\eqref{eq:partialtF} we easily get
\[
\partial_t ( \vert \nabla F \vert^2 ) =  - \langle \nabla \Delta F  , \nabla F  \rangle 
 - \langle \nabla \vert \nabla F \vert^2 , \nabla F \rangle . 
\]
Applying It\^o's formula we obtain after some computations 
(omitting variables in the right hand side)
\[
d \vert \nabla F (t,Y_t)\vert^2
=  \langle \nabla \vert \nabla F \vert^2 ,  \Phi_t d B_t \rangle 
 - \langle \nabla \Delta F , \nabla F \rangle \, dt 
+ \frac 12 \Delta \vert \nabla F \vert^2 \, dt  . 
\]
Now recall Bochner's formula
\[
\frac 12 \Delta | \nabla F |^2 = \langle \nabla \Delta F , \nabla F \rangle
 + \| \nabla^2 F \|_{HS}^2 + \mathrm{Ric} ( \nabla F , \nabla F ) .  
\]
So that
\begin{equation}
\label{eq:steplogsob}
d \vert \nabla F (t,Y_t)\vert^2
=  \langle \nabla \vert \nabla F \vert^2 ,  \Phi_t d B_t \rangle  
+ \Vert \nabla^2 F \Vert_{HS}^2 \, dt  + \mathrm{Ric} ( \nabla F , \nabla F ) \, dt .
\end{equation}
Since $\nabla F$ is bounded and the Ricci curvature non negative,
the local martingale part in the above equation is bounded from 
above. So by Fatou's lemma it is a sub--martingale,
and its expectation is non decreasing. So taking expectation
in~\eqref{eq:steplogsob} and using the hypothesis 
$\mathrm{Ric} \geq \kappa \, g$ we get
\begin{equation}\label{eq:steplogsob2}
\frac d{dt} \, \E\left[ \vert \nabla F_t (Y_t) \vert^2 \right]
\geq \E \left[ \Vert \nabla^2 F_t (Y_t)  \Vert_{HS}^2 \right]  
+ \kappa \, \E\left[ \vert \nabla F_t (Y_t) \vert^2 \right] . 
\end{equation}
Throwing away the Hessian term
would lead us to the inequality~\eqref{eq:logsob0}.  
Let us exploit this term instead. 
Using Cauchy-Schwartz and Jensen's inequality we get
\[
\E \left[ \Vert \nabla^2 F_t (Y_t) \Vert_{HS}^2 \right] \geq 
\frac 1 n \E \left[ \Delta F_t (Y_t)^2 \right]
\geq \frac 1 n  \, \E [ \Delta F_t (Y_t) ]^2. 
\]
Also, by~\eqref{eq:lawY} and recalling that $(P_t)$ is
the heat semigroup on $M$ we obtain 
\[
\begin{split}
\E \left[  \Delta F_t (Y_t) + | \nabla F_t ( Y_t ) |^2 \right]
& = \E \left[ \frac { \Delta P_{T-t} f ( Y_t ) }{ P_{T-t} f (Y_t) } \right] 
  = \E \left[ \frac { \Delta P_{T-t} f ( X_t ) }{ P_{T-t} f (X_t) } f(X_T) \right] \\
& = \E \left[ \Delta P_{T-t} f (X_t) \right] = \Delta P_T (f) (x) . 
\end{split}
\]
Letting $\alpha (t) = \E \left[ \vert \nabla F_t ( Y_t ) \vert^2 \right]$
and $C_T = \Delta P_T f (x)$ we thus get from~\eqref{eq:steplogsob2} 
\[
\begin{split}
\alpha'(t) 
& \geq \kappa \alpha (t) + \frac 1n ( \alpha (t) - C_T )^2  \\
& \geq  \frac 1 n \, \alpha(t) \, ( \alpha(t) +  n \, \kappa - 2 C_T  ) 
\end{split}
\]
Since $P_T f$ tends to a constant function as $T$ tends to $+\infty$, 
$C_T$ tends to $0$. So if $T$ is large enough $n \kappa - 2 C_T$ 
is positive and the differential inequality above yields
\[
\alpha(t) \leq \frac { n \, \kappa (T) \, \alpha (T) }
{ \e^{ \kappa(T) (T-t ) } \left( n \, \kappa(T) + \alpha(T) \right) - \alpha (T) } , \quad t \leq T 
\]
where $\kappa(T) = \kappa - 2 C_T / n$. 
Integrating this between $0$ and $T$ we get
\begin{equation}\label{eq:steplogsob3}
\int_0^T \alpha (t) \, dt \leq n \, \log \left( 1 + 
 \frac { \alpha(T) (1-\e^{ - \kappa(T) T}) }{ n \,  \kappa(T) } \right) .  
\end{equation}
Observe that $\kappa(T)\to \kappa$ as $T$ tends to 
$+\infty$. By~\eqref{eq:entropylogsob} and since
$(\delta_x P_t)$ converges to $m$
measure on $M$ we have
\[
\int_0^T \alpha (t) \, dt
\to 2\, \mathrm H ( \mu \mid m) ,
\]
as $T$ tends to $+\infty$. 
Also, since $Y_T$ has law $\mu$
\[
\alpha (T) = \E \left[ \vert \nabla \log f ( Y_T ) \vert^2 \right] =
\mathrm I ( \mu \mid \delta_x P_T ) \to \mathrm I ( \mu \mid m ). 
\]
Therefore, letting $T$ tend to $+\infty$ in~\eqref{eq:steplogsob3} yields
\[
\mathrm H ( \mu \mid m ) 
\leq \frac n 2  \log \left( 1 + \frac { \mathrm I ( \mu \mid m )}{ n \, \kappa } \right) , 
\]
which is the result.  
\end{proof}
Let us give an open problem to finish this article.
We already mentioned that Borell recovered 
the Pr\'ekopa--Leindler inequality 
from~\eqref{eq:borell}. It is
natural to ask whether there is probabilistic proof of
the Riemannian Pr\'ekopa--Leindler inequality of Cordero, McCann
and Schmuckenschläger~\cite{cordero} based on Theorem~\ref{thm:manifold}. 
Copying naively Borell's argument, we soon face the following difficulty:
If $X$ and $Y$ are two Brownian motions on a manifold 
coupled by parallel transport, then unless the manifold is flat, 
the midpoint of $X$ and $Y$ is not a Brownian motion. 
We believe that there is a way around this but we could
not find it so far.

\end{document}